\documentclass{amsart}
\usepackage{amsmath}

\usepackage{ytableau,color}
\usepackage{cite,array}
\usepackage{tikz}
\usepackage[all]{xy}
\usepackage{amssymb, amsmath,amsthm,color,amsfonts, mathrsfs}
\newtheorem{theorem}{Theorem}[section]
\newtheorem{definition}{Definition}[section]
\newtheorem{lemma}{Lemma}[section]
\newtheorem{corollary}{Corollary}[section]
\newtheorem{proposition}{Proposition}[section]
\newtheorem{conjecture}{Conjecture}[section]
\newtheorem{remark}{Remark}[section]
\newtheorem{example}{Example}[section]
\newtheorem{problem}{Problem}[section]

\title{Toward quantum Pieri rule for $F\ell_n$ via Seidel representation }
\author{Changzheng Li}
 \address{School of Mathematics, Sun Yat-sen University, Guangzhou 510275, P.R. China}
\email{lichangzh@mail.sysu.edu.cn}

\author{Jiayu Song}
 \address{School of Mathematics, Sun Yat-sen University, Guangzhou 510275, P.R. China}
\email{songjy29@mail2.sysu.edu.cn}

\thanks{This is an English translation of the same paper in Chinese, published at SCIENTIA SINICA Mathematica, 2024, 54(12): 2009-2022.}

\begin{document}

\maketitle

\begin{abstract}
     By using  a ``quantum-to-classical" reduction formula on  the Gromov-Witten invariants of flag vaireities $F\ell_n$, we provide a new proof of the Seidel operator on the quantum cohomology ring $QH^*(F\ell_n)$. Further, we reprove    a quantum Pieri rule with respect to certain special Schubert class for  $QH^*(F\ell_n)$. Finally, we propose a concrete conjecture on the corresponding quantum Pieri rule for the quantum $K$-theory of $F\ell_n$.
\end{abstract}

   \section{
Introduction }

The (big/small) quantum cohomology    $QH^{*}(X)$    of the complex projective manifold    $X$    is a deformation of its classical cohomology ring $H^{*}(X).$ Gromov-Witten invariants of genus 0 are used to define the quantum product of  $QH^{*}(X)$, which virtually compute the  number of rational curves (or pseudo-holomorphic curves, from the perspective of symplectic geometry) that satisfy appropriate  conditions.  The study of    $QH^*(X)$    has been a very popular research field  since the notion of quantum cohomology is introduced.

Classical cohomology    $H^*(\cdot)$    is a contravariant functor. Morphisms between topological spaces  $f:X \rightarrow Y$  naturally induce the ring homorphism $f^{*}: H^*(X) \rightarrow H^*(Y)$. However, quantum cohomology is different from classical cohomology, with the lackness of functoriality in general case.   Therefore, geometric objects have to be studied individually in general. This is one of the important reasons that make the study of quantum cohomology extremely difficult. For some cases, we can still  discuss functoriality of quantum cohomology appropriately. For example, there is a famous \textit{crepant resolution conjecture}:  for    $K$-equivalent smooth projective varieties (or orbifolds, Deligne-Mumford stacks) $Y_{+},Y_{-}$, (that is, there exist birational morphism $f_{\pm}: X\to Y_\pm$ such that $f_{+}^{*}K_{Y_+}\cong f_{-}^{*}K_{Y_-}$,) the corresponding quantum cohomologies $QH^*(Y_-)$, $QH^*(Y_+)$  should be related through the analytic continuation of quantum parameters. This conjecture was first proposed by Yongbin Ruan   \cite{Ruan}, and further developed by Bryan-Graber, Coates-Iritani-Tseng, Iritani and Ruan  \cite{BrGr,CIT09,CR13,Ir09}. The conjecture is a widespread concerning question, for which there are many progress, such as    \cite{CLLZ,LLW,CIJ18,GoWo}.

For the natural projection map between (partial) flag varieties, we can also talk about the functoriality of quantum cohomology appropriately.  Flag  varieties   $G/P$  are a class of projective manifolds with very nice properties, where    $G$    is a connected complex semisimple Lie group and    $P$    is a parabolic subgroup of    $G$. The classical cohomology    $H^*(G/P)$    has a natural    $\mathbb{Z}$-graded algebraic structure. Taking the Borel subgroup    $B\subset P$    of    $G$, we have a natural projection map   $\pi: G/B\to G/P$    from the complete flag  variety $G/B$    to the partial flag  variety  $G/P$. From the Leray-Serre spectral sequence, there is a $\mathbb{Z}^2$-graded algebra isomorphism  $H^*(G/B)\cong H^*(G/P)\otimes H^*(P/B) $. Further, we  take the parabolic subgroup    $P'$    that satisfies    $B\subset P'\subset P$    and obtain the corresponding fiber bundle    $P/B\to P/P'$    as well as a    $\mathbb{Z}^2$-graded algebraic structure on    $H^*(P/B)$. Combining them with the graded structure induced by    $G/B\to G/P$, we establish a    $\mathbb{Z}^3$-graded algebraic structure on    $H^*(G/B)$. In this way, we  at most obtain a    $\mathbb{Z}^{r+1}$-graded algebraic structure on    $H^*(G/B)$, where    $r$    is the semi-simple rank of the Levi subgroup of    $P.$    And the induced morphism    $\pi^*: H^*(G/P)\to H^*(G/B)$    is an injective homomorphism, which can be regarded as part of the isomorphism of this graded algebra (in the form of    $ \{ \alpha\otimes 1^{\otimes^r} \} _{\alpha\in H^*(G/P)}$   ).
In  \cite{LL2010, Li15} , Leung and the first named author   used the Peterson-Woodward comparison formula    \cite{Peterson1997, Woodward2005}    to define a    $\mathbb{Z}^{r+1}$-graded vector space structure on    $QH^*(G/B),$   and further proved that  $QH^*(G/B)$  is a    $\mathbb{Z}^{r+1}$ -filtered algebra under this graded structure. Moreover, its induced    $\mathbb{Z}^{r+1}$-graded algebra (after localization) is isomorphic to the tensor product of    $QH^*(G/P)$    and    $r$    quantum cohomologies of the form    $QH^*(P'/P'')$.  In this way, a quantum version of the Leray-Serre spectral sequence is given.  This graded algebra has very nice applications, especially on the "quantum    $\to$    classical" reduction principle. That is, 3-pointed genus zero Gromov-Witten invariants of $G/P$ with high degree  can be reduced to the classical intersection number    $G/B$   under certain conditions. In this ``quantum    $\to$    classical" principle, we further obtained the applications on quantum Pieri rules    \cite{LiHu15, LL2013} ,  which extended  the quantum Pieri rule of Ciocan-Fontanine    \cite{Ciocan1999}     and the related work of Buch, Kresch and Tamvakis     \cite{Ciocan1999, KT2003, KT2004, BKT2003, BKT2}.

The (quantum) cohomology of the flag varieties  $SL(n,\mathbb{C})/P$    has a  canonical additive basis of Schubert classes   $\sigma^u$. In the quantum product of Schubert classes,
   $$  \sigma^{u}\star \sigma^{v}=\sum_{\lambda_{P}, w}N_{u,v}^{w,\lambda_{P}}q_{\lambda_{P}} \sigma^{w},
  $$    the Schubert structure constant    $N_{u,v}^{w,\lambda_{P}}$    is a genus 0, 3-pointed Gromov-Witten invariant of    $G/P$    with an enumerative meaning. In particular, it is a non-negative integer. When    $P$    is a maximal parabolic subgroup,    $SL(n,\mathbb{C})/P=Gr(k, n)= \{ V\leq \mathbb{C}^n\mid \dim V=k \} $    is called a complex Grassmannian. The corresponding Schubert class can be labeled by a partition.    $\sigma^{u}=\sigma^{\mu}$ , where the partition    $\mu=(\mu_1, \cdots, \mu_k)=(u(k)-k, \cdots, u(2)-2, u(1)-1)\in \mathbb{Z}^k$    satisfies    $n-k\geq \mu_1\geq\cdots\geq \mu_k\geq 0$. We usually abbreviate the special partitions    $p=(p, 0, \cdots, 0)$ ,    $1^m=(1, \cdots, 1, 0, \cdots 0)$    (    $m$  copies of 1). These two special partitions are equivalent in the sense of    $Gr(k, n)\cong Gr(n-k, n)$.   The multiplication formula      $\sigma^p\star \sigma^\nu$  is called the quantum Pieri rule, which was first given by Bertram    \cite{Bertram1997}. The Seidel operator    \cite{Seid}       $\sigma^{1^k}\star$    generates a cyclic  group  $\mathbb{Z}/n\mathbb{Z}$ action on    $QH^*(Gr(k, n))$
   \cite{Belk, Post}, and then Belkale provided a new proof of the quantum Pieri rule using this group action. This approach is also directly generalized to the quantum $K$-theory for  Grassmannians  \cite{LLSY, BCP23}. For the quantum cohomology of flag varieties  $G/P$ of general Lie-type, the corresponding Seidel operator was studied in  \cite{CMP}. In this paper, we will follow this idea to re-study the quantum Pieri rule of the quantum cohomology    $QH^*(F\ell_n)$    of the complete flag variety  $F\ell_n=SL(n, \mathbb{C})/B$. That is, we hope to show
 {    \upshape       $$\mbox{ Quantum Pieri rule} = \mbox{ classical Pieri rule } + \mbox{ Seidel operator action.}$$      }  To be more precise, we consider the Schubert class $\sigma^{s_1s_2\cdots s_{n-1}}$ of $H^*(F\ell_n)$, which is the image of the special Schbuert class  $\sigma^{(1, \cdots, 1)}$ in    $H^*(Gr(n-1, n))$  of the natural monomorphism  $H^*(Gr(n-1, n))  \longrightarrow  H^*(F\ell_n) $. Here  $s_i=(i, i+1)$    is a transposition of the permutation group    $S_n$. We use the "quantum    $\to$    classical" reduction principle to give a precise characterization of quantum product with a Seidel operator    $\mathcal{T}$    of    $QH^*(F\ell_n)$    in \textbf{Theorem    \ref{Seidelelement}}, where  $\mathcal{T}$ is defined by
   $$\mathcal{T}(\sigma^u):=\sigma^{s_1s_2\cdots s_{n-1}}\star \sigma^u.$$     Combining it with the classical Pieri rule    \cite{Sottile}, we re-prove the quantum Pieri rule with respect to the aforementioned special  Schubert class in \textbf{Theorem    \ref{thmQPR}}.  We will define    $u\uparrow i:=(s_1s_2\cdots s_{n-1})^iu$    and    $\lambda(u), \lambda(u, i)$    in Section 3.1. Using these notations, Theorem    \ref{Seidelelement}    and Theorem    \ref{thmQPR}    can be combined and described as follows.
   \begin{theorem}   \label{mainthmQH}    Let    $1\leq m\leq n-1$    and    $u \in S_{n}$, Note    $k:=n-u(n)$. In    $QH^*(F\ell_{n})$, we have
    \begin{align}
      \mathcal{T}(\sigma^u)&=q_{\lambda(u)} \sigma^{u\uparrow 1} \\
         \sigma^{s_{n-m}  \cdots s_{n-2}s_{n-1}}* \sigma^{u}&=q_1^{-1}q_{2}^{-2}\cdots q_{n-1}^{1-n}q_{\lambda(u, k)}\mathcal{T}^{n-k}(\sigma^{s_{n-m}  \cdots s_{n-2}s_{n-1}}\cup \sigma^{u\uparrow k}).
     \end{align}     \end{theorem}

It is our main motivation to the study of  quantum Pieri rule at the quantum    $K$-theoretical level, by  interpreting  the quantum Pieri rule of
$QH^*(F\ell_n)$ in the above way.  In section 3.3, we propose \textbf{Conjecture \ref{conjQK}}. That is, we
should be able to obtain the quantum $K$-theoretical one, by simply replacing the Schubert cohomology class ``$\sigma$" in Theorem    \ref{mainthmQH}    with the Schubert class ``$\mathcal{O}$" in   the $K$-theory $K(F\ell_n)$. Namely the study of the corresponding  quantum Pieri rule for the  quantum $K$-theory $QK(F\ell_n)$ should be reduced to the  Pieri rule for the   $K$-theory  $K(F\ell_n)$ obtained by  Lenart-Sottile   \cite{LeSo}. This could help us to understand the general quantum Pieri rule    \cite{NaSa}    on  $QK(F\ell_n)$. Moreover, the quantum $K$-theory of the flag varieties  $SL(n, \mathbb{C})/P$  admits  functoriality induced by  the natural projection map between flag varieties \cite{BCLM, Kato}.  With the help of  the induced surjective algebra homomorphism, we can understand the quantum Pieri rule of special Schubert class index by $s_is_{i+1}\cdots s_{j}$ with general Schubert class in quantum $K$-theory of non-complete flag varieties  $SL(n, \mathbb{C})/P (P \neq B)$.  We provide Example \ref{exevidence} for $F\ell_6$ and a  Pieri-type product of for the   quantum $K$-theory of $Gr(3,6)$  induced by Conjecture \ref{conjQK}, which is consistent with the quantum Pieri rule    \cite{BuMi}  obtained by Buch-Mihalcea. This provides an evidence for our Conjecture    \ref{conjQK}.

\subsection*{Acknowledgement} The authors are supported in part by the National Key Research and Development Program of China No.~2023YFA100980001 and NSFC  12271529.

   \section{
Functoriality of quantum cohomology of flag varieties  }    In this section, we briefly review the functoriality of quantum cohomology of  flag varities  in the       series of work \cite{LL2010,LL2012,Li15}   and its application on the reduction of ``quantum    $\to$    classical". On the one hand, our statement  will only focus on  flag varieties of type  $A_{n-1}$, which is very  concrete. On the other hand, in this section we will use the standard notation in Lie theory to indicate that the corresponding results hold true for all Lie types.

   \subsection{
Notations  }    We introduce commonly used notations in Lie theory. For more details, please refer to  \cite{Hump75}. Consider complex simple Lie group $G=SL(n,\mathbb{C})$. Let $B$ be the standard Borel subgroup of $G$, consisting of upper-triangular matrices, and  $P $ be a parabolic subgroup of $G$ containing $B$.
Denote  by $\mathfrak{h}$    the Lie algebra of the Lie subgroup    $T$  that consists of   the diagonal matrices of    $G$. Let   $\Delta= \{ \alpha_1,\alpha_2,\cdots, \alpha_{n-1} \}  \subset \mathfrak{h}^{*}$    be the standard simple  roots, and $\{\alpha_1^{\vee},\alpha_2^{\vee},\cdots, \alpha_{n-1}^{\vee}\} \subset \mathfrak{h}$ be the simple coroots. Denote by $
\langle \cdot,\cdot \rangle: \mathfrak{h}^* \times \mathfrak{h} \longrightarrow \mathbb{C} $   the natural pairing.  Let $Q^{\vee}=\oplus_{i=1}^{n}\mathbb{Z}\alpha_i^{\vee}$  and $\rho=\sum_{i=1}^{n}\chi_{i} \in \mathfrak{h}^{*}$, where $\chi_{i}$ are the fundamental weights satisfying $\langle \chi_i, \alpha_j^{\vee}\rangle=\delta_{i,j}$.
 The Weyl group     $W$    of    $G$    is generated by the simple reflection    $ \{ s_i:=s_{\alpha_i}\mid 1\leq i\leq n-1 \} $    and is isomorphic to the permutation group    $S_n$.  Here the simple reflection  $s_i: \mathfrak{h}^*\to \mathfrak{h}^*; s_i(\beta)=\beta-\langle\beta, \alpha_{i}^{\vee} \rangle\alpha_{i}$ corresponds to
 transposition $(i, i+1)$ in $S_n$. We freely interchange $s_i$ and $(i, i+1)$ whenever there is no confusion. There is a standard length function (with respect to the generators    $\{ s_i \} _i$) on the Weyl group, denoted a
 s $\ell: W \rightarrow \mathbb{Z}_{\geq 0}$.
 The parabolic subgroup $P\supset B$ corresponds to a unique subset   $\Delta_P=\Delta\setminus\{\alpha_{n_1}, \cdots, \alpha_{n_k}\}$  of $\Delta$, where $1\leq n_1<\cdots<n_k\leq n-1$.

    \begin{enumerate}   \item    [1)]
    $\Delta_B=\emptyset$. Let    $P_{\alpha_i}$    be the parabolic subgroup corresponding to the subset    $ \{ \alpha_i \} $, then    $$P_{\alpha_i}= \{ (g_{ab})\in SL(n, \mathbb{C})\mid g_{ab}=0, \mbox{ if } a>b \mbox{ and } (a, b)\neq (i+1, i) \}.$$

    \item    [2)]    $r:=|\Delta_P|$    is the semisimple rank of the Levi subgroup of    $P$.

    \item    [3)] The root system   $R$    can be obtained from the action of the Weyl group  on the set   $\triangle$ of simple roots:  $R=W\cdot \triangle=R^{+}\sqcup(-R^{+})$, where \ $R^{+}=R \cap \oplus_{i=1}^{n-1} \mathbb{Z}_{\geq 0}\alpha_{i}$ is the set of positive roots  respect to $\triangle$. Denote  $R^{+}_P:=R \cap \oplus_{\alpha\in \Delta_P} \mathbb{Z}_{\geq 0}\alpha$, $Q_{P}^{\vee}:=\oplus_{\alpha\in \Delta_{P}}\mathbb{Z}\alpha^{\vee}$.

    \item    [4)] Let    $W_{P}$    be the Weyl subgroup generated by    $ \{ s_{\alpha}|\alpha\in \Delta_{P} \} $, Then    $W^{P}:= \{ w\in W| \ell(w) \leq \ell(v), \forall v\in wW_{P} \} \subset W$    is the set of minimum length representatives of    $W/W_P$. Denote   $n_0:=0, n_{k+1}:=n$. As a subset of $S_{n}$, we have
    $$W^P= \{ w\in S_n\mid w(n_{i-1}+1)<w(n_{i-1}+2)<\cdots <w(n_i), \quad\forall 1\leq i\leq k+1 \}.$$
    \item    [5)]   There is a unique longest element in  $W$ (resp. $W_P$), denoted as $w_0$ (resp. $w_P$). As a permutation of $S_n$, $w_0(j)=n+1-j$, $\forall 1\leq j\leq n$.

     \item[6)] $W\cong N(T)/T$, here $N(T)\leq G$ is the normalizer of $T$ in $G$. For $u\in W$, we denote by $\dot u\in N(T)$   a representative of the corresponding coset set in $N(T)/T$  under this standard isomorphism.
     \end{enumerate}

   \subsection{
Quantum cohomology}    The content of this section mainly follows from   \cite{FP1997, FW2004}. Flag varieties   $G/P$   parameterize  partial  flags of specific type  in    $\mathbb{C}^n$:
    $G/P= \{ V_{n_1}\leq \cdots \leq V_{n_k}\leq \mathbb{C}^n\mid \dim V_{n_j}=n_j, \forall ~j \} =:F\ell_{n_1, \cdots, n_k; n}$    In particular,    $G/B=F\ell_{1, 2, \cdots, n-1; n}=:F\ell_n$    is a complete flag variety. The (opposite) Schubert varieties     $X_u$, $X^u$ in    $G/P$    are respectively defined by
    $$X_u=\overline{B\dot u P/P},\quad X^u=\overline{\dot w_0B\dot w_0\dot uP/P}.$$
     Thus the complex dimension of  $X_u$ (resp.   codimension of  $X^u$) is   $\ell(u)$. The Schubert  (co)homology   classes    $\sigma_u:=[X_u]\in H_{2\ell(u)}(G/P, \mathbb{Z})$    and    $\sigma^u_P:=P.D.[X^u]\in H^{2\ell(u)}(G/P, \mathbb{Z})$    respectively form  an additive bases    \cite{BGG1973}    of the (co)homology of the flag varieties    $X=G/P$:
    $$H_{*}(G/P,\mathbb{Z})=\bigoplus_{u\in W^{P}}\mathbb{Z} \sigma_{u},H^{*}(G/P,\mathbb{Z})=\bigoplus_{u\in W^{P}}\mathbb{Z} \sigma^{u}_P.$$    Note that the natural projection map    $\pi: G/B\to G/P$  induces a  monomorphism
    $$\pi^*: H^*(G/P, \mathbb{Z})\to H^*(G/B, \mathbb{Z}); \pi^*(\sigma^u_P)=\sigma^u_B.$$    Therefore we abbreviate    $\sigma^u_P, \sigma^u_B$    as    $\sigma^u$. Besides, we have the canoincal isomorphism    $H_{2}(G/P,\mathbb{Z})=\bigoplus_{\alpha_{i}\in \Delta - \Delta_{P}}\mathbb{Z} \sigma_{s_{i}} \cong Q^{\vee}/Q_{P}^{\vee}$ of abelian groups.

Given    $\lambda_P\in H_2(G/P, \mathbb{Z})=Q^\vee/Q^\vee_P$, we denote  by  $\overline{\mathcal{M}_{0,3}}(G/P,\lambda_{P})= \{ (f: \mathbb{P}^{1}\to G/P; p_1, p_2, p_3)\mid f_*([\mathbb{P}^{1}])=\lambda_P, f\mbox{ is a stable map}  \} $      the moduli space of 3-pointed, genus-zero stable maps of degree $\lambda_P$. This moduli space is an orbifold, and its dimension is equal to    $\dim G/P + \langle c_{1}(G/P), \lambda_{P}\rangle$. Let    $ev_{i}:\overline{\mathcal{M}_{0,3}}(G/P,\lambda_{P}) \longrightarrow G/P $    be the    $i$-th   evaluation map. For   $\alpha_{n_j}\in \Delta - \Delta_{P}$, we introduce the formal variable  $q_{n_j}$. For    $\lambda_{P}=\sum_{\alpha_{n_j}\in \Delta - \Delta_{P}}b_{j} \alpha_{n_j}^{\vee} +Q_{P}^{\vee}$, we denote    $q_{\lambda_{P}}=\prod_{\alpha_{n_j}\in \Delta - \Delta_{P}} q^{b_{j}}_{n_j}$, The (small) quantum cohomology    $QH^*(G/P)=(H^*(G/P)\otimes \mathbb{C}[q_{n_1}, \cdots, q_{n_k}], \star)$    of the flag variety $G/P$ is defined by
    $$\sigma^{u}\star \sigma^{v}=\sum_{\lambda_{P} \in H_{2}(G/P,\mathbb{Z}),w\in W^{P}}N_{u,v}^{w,\lambda_{P}}q_{\lambda_{P}} \sigma^{w}. $$    Here  the quantum Schubert structure constant    $N_{u,v}^{w,\lambda_{P}}$    is a genus-0, 3-pointed Gromov-Witten invariant of degree    $\lambda_P$    on    $G/P$, given by the following integral
    $$N_{u,v}^{w,\lambda_{P}}=\int_{\overline{\mathcal{M}_{0,3}}(G/P,\lambda_{P})}ev_1^{*}(\sigma^{u}) \cup ev^{*}_{2}(\sigma^{v}) \cup ev^{*}_{3}(\sigma^{w^{\vee}}),$$    where    $w^\vee:=w_0ww_P\in W^P$. Geometrically, for    $g, g'\in G$    in general position, we have
   \begin{align}\label{positivity}
   N_{u, v}^{w, \lambda_P}=\sharp\{f:\mathbb{P}^1\to G/P\mid f_*([\mathbb{P}^1])=\lambda_P, f(0)\in X^u, f(1)\in g X^v, f(\infty)\in g'X^{w^\vee}\}\in \mathbb{Z}_{\geq 0}\mbox{. }
  \end{align}
   By the  definition of the moduli space  and its dimension formula, we have
    \begin{equation}\label{degreeaxiom}
     N_{u,v}^{w,\lambda_{P}}=0 ~\text{ unless}~ \ell(u)+\ell(v) = \ell(w)+ \langle  c_{1}(G/P), \lambda_{P}\rangle \mbox{ and }\lambda_P\geq \boldsymbol{0}\mbox{. }
   \end{equation}    Here    $\lambda_P\geq \mbox{0}$    means    $b_j\geq 0$ ,    $\forall~ 1\leq j\leq k$. Since    $c_1(G/P)>0$ , the right side of the quantum product is a finite sum. Therefore,
    \begin{enumerate}   \item    [1)]    $QH^*(G/P)=\bigoplus \mathbb{C}q_{\lambda_P}\sigma^w$    is a     $\mathbb{Z}$-graded algebra, where the    $\mathbb{Z}$-graded structure is naturally given by the degree of the basis    $\{q_{\lambda_P}\sigma^w\}$   :
    $$\deg q_{\lambda_P}\sigma^w=\ell(w)+ \langle  c_{1}(G/P), \lambda_{P}\rangle. $$
In particular  for    $G/B$, we have    $\langle  c_{1}(G/B), \lambda\rangle=\langle 2\rho, \lambda\rangle$, where    $\lambda\in H_2(G/B, \mathbb{Z})=Q^\vee$.
    \item    [2)]    $\sigma^{u}\cup \sigma^{v}=\sigma^{u}\star \sigma^{v} |_{\boldsymbol{q}=\boldsymbol{0}}$. That is, the quantum cohomology    $QH^*(G/P)$    is a deformation of the classical cohomology    $H^*(G/P)$. \end{enumerate}

   \subsection{Functoriality of quantum cohomology  }
   In this section, we briefly explain the    $\mathbb{Z}^2$-graded algebra structure on    $QH^*(G/B)$    introduced by    \cite{LL2010,LL2012,Li15}    for the special case of natural projection map  $\pi: G/B\to G/P_\alpha$    ($\alpha\in \Delta$) (note that    $r=1$). This structure can be regarded as  an induced ``morphism" at the quantum cohomology level from $\pi$, so we roughly call it ``functoriality".

Note that when    $\alpha=\alpha_i$, the projection map    $\pi$    is the natural forgetful  map    $F\ell_n=G/B\to G/P_{\alpha_i}=F\ell_{1,2, \cdots, i-1, i+1, \cdots, n-1; n}$, which is a fiber bundle with fiber $P_{\alpha_i}/B\cong \mathbb{P}^1$. As vector spaces, $QH^*(G/B)=H^*(G/B)\otimes \mathbb{C}[q_1, \cdots, q_{n-1}]$. We consider  the basis $\{q_{\lambda}\sigma^{w}|(w,\lambda)\in W \times Q^{\vee}\}$ of  the localization  $QH^{*}(G/B)[q_{1}^{-1}, q_{2}^{-1}, \cdots, q_{n-1}^{-1}]$ of $QH^*(G/B)$, and introduce a map $\mbox{sgn}_{\alpha}$ with respect to a  given simple root $\alpha \in \triangle$:
    $$\mbox{sgn}_{\alpha}: W \longrightarrow  \{ 0,1 \} ; \quad \mbox{sgn}_{\alpha}(w)=\left \{
\begin{aligned}
1,~  \ell(w)-\ell(ws_{\alpha}) > 0, \\
0,~  \ell(w)-\ell(ws_{\alpha}) \leq 0.
\end{aligned}
\right. $$    Note that    $\ell(w)-\ell(ws_{\alpha})=\pm 1$ , and    $\ell(w)-\ell(ws_{\alpha})=1$  holds  if and only if     $u:=ws_\alpha \in W^{P_{\alpha}}$. Then we can define a   $\mathbb{Z}^2$-grading map with respect to $\alpha$ as follows:
    $$gr_{\alpha}: W \times Q^{\vee} \longrightarrow \mathbb{Z}^{2};$$
    $$gr_{\alpha}(q_{\lambda} \sigma^{w})=(\mbox{sgn}_{\alpha}(w)+\langle\alpha, \lambda\rangle, \ell(w)+\langle 2\rho, \lambda \rangle-\mbox{sgn}_{\alpha}(w)- \langle\alpha, \lambda\rangle ).$$
    We notice
    \begin{enumerate}   \item    [a)]    $q_{\lambda} \sigma^{w}\in QH^{*}(G/B)$    (resp.   $QH^{*}(G/B)[q_i^{-1}]$) if and only if    $q_\lambda\in \mathbb{C}[q_1, \cdots, q_{n-1}]$    (resp.  $\mathbb{C}[q_1, \cdots, q_{n-1}][q_i^{-1}]$).
    \item    [b)] Denote    $gr_{\alpha}(q_{\lambda} \sigma^{w})=(a, b)$, then we have    $\deg(q_{\lambda} \sigma^{w})=a+b$. Therefore    $gr_\alpha$    is actually  a  $\mathbb{Z}^2$-graded refinement    of     the  $\mathbb{Z}$-graded structure    of    $QH^*(G/B)$    in Section 2.2.  \end{enumerate}

We use the  lexicographical order on    $\mathbb{Z}^{2}$. That is,    $ a =(a_1,a_2) < b =(b_1,b_2)$    if and only if   either $a_{1} < b_1$    or ($a_1=b_1$ and $a_2 < b_2$) holds. Then we define $$F_{\textbf{a}}:=\bigoplus\limits_{gr_{\alpha}(q_{\lambda} \sigma^{w}) \leq \textbf{a}} \mathbb{C}q_{\lambda} \sigma^{w} \subset QH^{*}(G/B),~ F_{\textbf{a}}':=\bigoplus\limits_{gr_{\alpha}(q_{\lambda} \sigma^{w}) \leq \textbf{a}} \mathbb{C}q_{\lambda} \sigma^{w} \subset QH^{*}(G/B)[q_{\alpha^\vee}^{-1}].$$
  As a consequence, we obtain a family  $\mathcal{F}=\{F_{\textbf{a}}\}_{\textbf{a}\in \mathbb{Z}^{2}}$ of vector subspaces of $QH^{*}(G/B)$ , and a family $\mathcal{F}'$ of vector subspaces of  localization $QH^{*}(G/B)[q_{\alpha^\vee}^{-1}]$  of $QH^{*}(G/B)$  (where  $\mathcal{F}'$ can be regarded as the natural extension of $\mathcal{F}$).
 Their associated  $\mathbb{Z}^2$-graded vector spaces are    $$ Gr^{\mathcal{F}}(QH^{*}(G/B))=\bigoplus_{ a \in \mathbb{Z}^{2}}Gr_{ a }^{\mathcal{F}},\qquad Gr^{\mathcal{F}'}(QH^{*}(G/B)[q_{\alpha^\vee}^{-1}])=\bigoplus_{ a \in \mathbb{Z}^{2}}Gr_{ a }^{\mathcal{F}'}$$
 respectively, where    $Gr_{ a }^{\mathcal{F}}:=F_{ a }/\cup_{ b  < a }F_{ b }$  and   $Gr_{\textbf{a}}^{\mathcal{F}'}:=F_{\textbf{a}}'/\cup_{\textbf{b} <\textbf{a}}F_{\textbf{b}}'$.
 The definition of the map $gr_\alpha$ is inspired by the following Peterson-Woodward comparison formula, which  is proposed by Peterson in his unpublished work
 \cite{Peterson1997}, and proved by Woodward \cite{Woodward2005} later. This formula  reduces the computation of the Gromov-Witten invariant  $N_{u,v}^{w,\lambda_{P}}$ of any (partial) flag variety $G/P$  to the computation of the corresponding Gromov-Witten invariant $N_{u,v}^{ww_{P}w_{P^{'}},\lambda_{B}}$ of the complete flag variety $G/B$. In general, the corresponding    $\mathbb{Z}^{r+1}$-graded map  is more complicated, for which we refer to    \cite{LL2010}  for the details.

    \begin{proposition}[Peterson-Woodward comparision formula]

    (1) Let    $\lambda_{P} \in Q^{\vee}/Q_{P}^{\vee},$    then there is a unique    $\lambda_{B}\in Q^{\vee}$    such that    $\lambda_{P}=\lambda_{B}+Q_{P}^{\vee}$    and for all    $\gamma \in R_{P}^{+}$ ,     $\langle  \gamma , \lambda_{B} \rangle \in  \{ 0,-1 \} $.

 (2) For any   $u,v,w\in W^{P}$, we have
    $$N_{u,v}^{w,\lambda_{P}}=N_{u,v}^{ww_{P}w_{P^{'}},\lambda_{B}}, $$    where    $P'$    is the parabolic subgroup corresponding to the subset    $\Delta_{P^{'}}= \{ \beta \in \triangle_{P}| \langle \beta, \lambda_{B} \rangle=0 \} $.  \end{proposition}

According to the Peterson-Woodward comparison formula, we get an injection map between complex vector spaces:
    $$\psi_{\alpha}:QH^{*}(G/P_\alpha) \longrightarrow QH^{*}(G/B),$$
    $$ q_{\lambda_{P_\alpha}}\sigma^{w} \longmapsto q_{\lambda_{B}}\sigma^{ww_{P}w_{P^{'}}}.$$

   \begin{proposition}[  Theorem 1.2 of \cite{LL2010}]   \label{propfilteralgebra}
    $QH^{*}(G/B)$    is   a $\mathbb{Z}^{2}$-filtered algebra with respect to    $\mathcal{F}$. That is,  we have   $$F_{ a } \star F_{ b } \subset F_{ a+b },$$ for any    $ a , b \in \mathbb{Z}^{2}$.     \end{proposition}

Denote   $Gr_{\rm ver}^{\mathcal{F}}(QH^*(G/B)):=\bigoplus\limits_{i\in \mathbb{Z}}Gr_{(i, 0)}^{\mathcal{F}}$    and
   $Gr_{\rm hor}^{\mathcal{F}}(QH^*(G/B)):=\bigoplus\limits_{j\in \mathbb{Z}}Gr_{(0, j)}^{\mathcal{F}}$. By replacing    $\mathcal{F}$    with    $\mathcal{F}'$,   we define this  similarly   for    $QH^*(G/B)[q_{\alpha^\vee}^{-1}]$. Furthermore, we consider the standard isomorphism    $QH^*(\mathbb{P}^1)\cong {\mathbb{C}[x, t]\over (x^{2}-t) }$ , where    $t$    labels the quantum variable of    $QH^*(\mathbb{P}^1)$.  As a core consequence of Proposition    \ref{propfilteralgebra} , we have

   \begin{corollary}[Theorem 1.4 of \cite{LL2010} ]   \label{propfiberisomor}    The following maps  {    \upshape       $\Psi_{\rm ver}^\alpha$      }  and  {    \upshape       $\Psi_{\rm hor}^\alpha$     }  are well-defined, and they are algebra isomorphisms.    \footnote{ In terms of  notations of    \cite{LL2010},    $\Psi_{\rm ver}^\alpha=\Psi_1$    and    $\Psi_{\rm hor}^\alpha=\Psi_2$.  }
  {    \upshape       $$\begin{array}{rrc}
       \Psi_{\rm ver}^\alpha:&  QH^*(\mathbb{P}^1) \longrightarrow  Gr_{\rm ver}^{\mathcal{F}}(QH^*(G/B)); &
              \quad x\mapsto            \overline{s_\alpha},\,\,  t\mapsto \overline{q_{\alpha^\vee}}\,\,\,. \\
         \Psi_{\rm hor}^\alpha:&  QH^*(G/P_\alpha) \longrightarrow  Gr_{\rm hor}^{\mathcal{F}}(QH^*(G/B)); &
                  \,\,     q_{\lambda_{P_\alpha}}\sigma^w\mapsto           \overline{\psi_{\alpha}(q_{\lambda_{P_\alpha}}\sigma^w)}\,\,\,.
\end{array}$$   Here   $\overline{s_\alpha} \in Gr^\mathcal{F}_{(1, 0)}\subset Gr_{\scriptsize\mbox{vert}}^\mathcal{F}(QH^*(G/B))$    dentoes  the graded component of    $\sigma^{s_\alpha}+\cup_{\boldsymbol{b}<(1, 0)}F_{\boldsymbol{b}}$.
 After a  natural extension , we have the following    $\mathbb{Z}^2$-graded algebra isomorphism:
   $$\Psi_{\rm ver}^\alpha\otimes \Psi_{\rm hor}^\alpha:  QH^*(\mathbb{P}^1)[t^{-1}]\otimes QH^*(G/P_\alpha)\overset{\cong}{\longrightarrow} Gr^{\mathcal{F}'}(QH^*(G/B)[q_{\alpha^\vee}^{-1}]). $$      }   \end{corollary}

   \begin{remark} The proofs of Proposition    \ref{propfilteralgebra}    and Corollary    \ref{propfiberisomor}    are mainly based on the non-negativity of the structure constant    $N_{u, v}^{w, \lambda_P}$, the Peterson-Woodward comparison formula and the following quantum Chevalley formula (for general    $G/P$, see    \cite{FW2004}), and  the induction on  $\ell(u)$.

    \begin{proposition}   \label{quanchevalley} Let   $u\in  W$ and  $1\leq i\leq n-1$.  In  $QH^*(G/B)$,     we have
          $$ \sigma^u\star\sigma^{s_i} =\sum \langle \chi_i, \gamma^\vee\rangle \sigma^{us_\gamma}+\sum  \langle \chi_i, \gamma^\vee\rangle q_{\gamma^\vee}\sigma^{us_\gamma},$$   the first  sum over positive roots          $\gamma\in R^+$          that satisfy          $\ell(us_\gamma)=\ell(u)+1$, and the second  sum over positive roots          $\gamma\in R^+$          that satisfy          $\ell(us_\gamma)=\ell(u)+1-\langle 2\rho, \gamma^\vee\rangle$.  \end{proposition}     \end{remark}

   \subsection{``Quantum    $\to$    Classical" reduction}

Corollary    \ref{propfiberisomor}    can give  applications of Gromov-Witten invariants in the ``quantum    $\to$    classical" reduction. We illustrate the main idea with a simple example.

   \begin{example}   \label{examFl3}    Consider    $G=SL(3, \mathbb{C})$,   $\alpha=\alpha_1$.  That is,    $G/B=F\ell_3$,    $G/P_\alpha=Gr(1, 3)= \mathbb{P}^2$. Then we have
    \begin{align*} C t\otimes \sigma_P^{s_1s_2}= & \   (\Psi_{\rm ver}^\alpha\otimes \Psi_{\rm hor}^\alpha)^{-1}(C \overline{q_1\sigma^{s_1s_2}_B}) \\
    = & \  (\Psi_{\rm ver}^\alpha\otimes \Psi_{\rm hor}^\alpha)^{-1}(\overline{\sigma^{s_2s_1}_B}\star\overline{\sigma^{s_2s_1}_B}) \\
   = &  \    (\Psi_{\rm ver}^\alpha\otimes \Psi_{\rm hor}^\alpha)^{-1}(\overline{\sigma^{s_2s_1}_B})\star (\Psi_{\rm ver}^\alpha\otimes \Psi_{\rm hor}^\alpha)^{-1}(\overline{\sigma^{s_2s_1}_B}) \\
   = & \   (x\otimes \sigma^{s_2}_P)\star (x\otimes \sigma^{s_2}_P) \\
    = & \   (x\star_{\mathbb{P}^1}x)\otimes (\sigma^{s_2}_P\star_{P}  \sigma^{s_2}_P) \\
    = & \   t\otimes D \sigma_{P}^{s_1s_2}.
 \end{align*}    Therefore, we have    $C=D$. Note that    $C=N_{s_2s_1, s_2s_1}^{s_1s_2, \alpha_1^\vee}$    is the Gromov-Witten invariant of degree    $\alpha^\vee_1$    in    $QH^*(G/B)$. $D= N_{s_2, s_2}^{s_1s_2, 0}$ is an  intersection number in the classical cohomology $H^*(G/P_\alpha)$, which can also be regarded as a classical intersection number in $H^*(G/B)$. Therefore we obtain ``quantum to classical" reduction:
    $$N_{s_2s_1, s_2s_1}^{s_1s_2, \alpha_1^\vee}=N_{s_2, s_2}^{s_1s_2, 0}.$$     \end{example}

From the definition of the graded mapping    $gr_\alpha$, we immediately obtain:
    \begin{lemma}   \label{lemmagradcomp}    Let    $u, v, w\in W$    and    $\lambda \in Q^\vee$. Then we have
    $gr_\alpha(\sigma^u)+gr_\alpha(\sigma^v)=gr_\alpha(q_\lambda\sigma^{w})$    if and only if the following two conditions both hold true:
    $${\rm (1)}\quad \ell(w)+\langle 2\rho, \lambda\rangle=\ell(u)+\ell(v),\qquad {\rm  (2)} \quad {\rm sgn}_\alpha(w)+\langle \alpha, \lambda\rangle ={\rm sgn}_\alpha(u)+{\rm sgn}_\alpha(v).$$       \end{lemma}

    Therefore, based on the idea of Example    \ref{examFl3}    and the above lemma, using  Corollary    \ref{propfiberisomor},   we can obtain the following reduction.

   \begin{proposition}[Theorem 1.1 of \cite{LL2012}]   \label{quantum-classical}    For any    $u, v, w\in W$    and any    $\lambda\in Q^\vee$, we have
    \begin{enumerate}   \item    [a)]          $N_{u, v}^{w, \lambda}=0$           unless   {    \upshape             $ \mbox{sgn}_\alpha(w)+\langle \alpha, \lambda\rangle \leq \mbox{sgn}_\alpha(u)+\mbox{sgn}_\alpha(v)$            } for all    $\alpha\in \Delta $.
    \item    [b)] If  {    \upshape             $ \mbox{sgn}_\alpha(w)+\langle \alpha, \lambda\rangle =\mbox{sgn}_\alpha(u)+\mbox{sgn}_\alpha(v)=2$            }  holds for some          $\alpha\in \Delta$, then
  {    \upshape             $$N_{u, v}^{w, \lambda}=N_{us_\alpha, vs_\alpha}^{w, \lambda-\alpha^\vee}=
                 \begin{cases} N_{u, vs_\alpha}^{ws_\alpha, \lambda-\alpha^\vee}, &       \mbox{sgn}_\alpha(w)=0 , \\
                               \vspace{-0.3cm}   &  \\
                         N_{u, vs_\alpha}^{ws_\alpha, \lambda}, &     \mbox{sgn}_\alpha(w)=1 \,\,.  \end{cases}$$            }   \end{enumerate}     \end{proposition}

   \begin{corollary}   \label{cored}
   If    $u, v, w\in W$    and    $\alpha\in \Delta $    satisfy    $\mbox{sgn}_\alpha(v)=\mbox{sgn}_\alpha(w)=1$    and    $\mbox{sgn}_\alpha(u)=0$, then    $$N_{u, v}^{w, 0}=N_{u, vs_\alpha}^{ws_\alpha, 0}.$$     \end{corollary}
   \begin{proof} Let    $\bar u=v, \bar v= us_\alpha, \bar w=ws_\alpha$ , From the proposition    $\mbox{sgn}_\alpha(\bar w)=0$    and    $\mbox{sgn}_\alpha(\bar w)+\langle \alpha, \alpha^\vee\rangle =\mbox{sgn}_\alpha(\bar u)+\mbox{sgn}_\alpha(\bar v)=2$ , Therefore from Proposition \ref{quantum-classical},
    $$N_{u, vs_\alpha}^{ws_\alpha, 0}=N^{ws_\alpha, 0}_{vs_\alpha, u}=N_{\bar u s_\alpha, \bar vs_\alpha}^{\bar w, \alpha^\vee-\alpha^\vee}=N_{\bar u, \bar vs_\alpha}^{\bar w s_\alpha, \alpha^\vee-\alpha^\vee}=N_{v, u}^{w, 0}=N_{u, v}^{w, 0}. $$     \end{proof}
   \begin{example} Consider    $G/B=F\ell_4$. Take    $u=s_3s_2s_1s_2$,    $v=s_2s_1s_2$,    $w=s_1s_2s_3$    and    $\lambda=\alpha_1^\vee+\alpha_2^\vee$.

    $$N_{u, v}^{w, \lambda}=N_{u, vs_3}^{ws_3, \lambda+\alpha_3^\vee} =N_{s_3s_2s_1s_2, s_2s_1s_2s_3}^{s_1s_2, \alpha_1^\vee+\alpha_2^\vee+\alpha_3^\vee},$$
    $$N_{u, v}^{w, \lambda}=N_{s_3s_2s_1 , s_2s_1s_2s_3}^{s_1, \alpha_1^\vee+\alpha_2^\vee+\alpha_3^\vee}
         =N_{s_3s_2s_1, s_2s_1s_2}^{s_1s_3, \alpha_1^\vee+\alpha_2^\vee}
         =N_{s_3s_2s_1,  s_2s_1}^{s_1s_3s_2, \alpha_1^\vee}
         =N_{s_3s_2,  s_2}^{s_1s_3s_2, 0}=1.$$     \end{example}

We call    $u\in S_n$    a Grassmannian type permutation if there exists    $k$    such that    $u(1)<u(2)<\cdots <u(k)$    and    $u(k+1)<u(k+2)<\cdots < u(n)$ , As an application of the ``quantum    $\to$    classical" reduction, we have

   \begin{proposition}[Theorem 1.2 of \cite{LL2012}] Let   $u, v, w\in S_n$    and    $\lambda\in Q^\vee$, If    $u$    is Grassmannian type permutation, then there is    $v', w'\in S_n$    such that in    $QH^*(F\ell_n)$,
    $$ N_{u, v}^{w, \lambda}=N_{u, v'}^{w', 0}.$$     \end{proposition}    In the next section, we will discuss  the special Grassmannian type permutation   $u=s_1s_2\cdots s_{n-1}$   in detail.

   \section{From Seidel representation to quantum Pieri rule }

In this section, we discuss how to use Seidel representation in    $QH^*(F\ell_n)$    to obtain the quantum Pieri rule of the form    $\sigma^{s_is_{i+1}\cdots s_{n-1}}\star \sigma^u$, and propose expectations at  quantum    $K$-theory level.

   \subsection{Seidel operator and quantum Pieri rule  }    In    $S_n$,    $s_1s_2\cdots s_{n-1}=(1,2,\cdots, n)$    is an    $n$-cyle. So    $\langle s_1s_2\cdots s_{n-1}\rangle \cong \mathbb{Z}/n\mathbb{Z}$    is a cyclic group of order    $n$. Denote   $u_j^{(m)}=s_{m-j+1}\cdots s_{m-1}s_m$       ($0\leq j\leq m$ )   and    $u_0^{(m)}:=\mbox{id}$. For any    $u\in S_n$, we have the following properties (see Corollary 2.6 in    \cite{LL2010}   ):
   \begin{enumerate}   \item    [(1)] There exists a unique    $(j_{n-1}, \cdots, j_2, j_1)$    such that    $u=u_{j_{n-1}}^{(n-1)}\cdots u_{j_2}^{(2)}u_{j_1}^{(1)}$    is a  reduced expression of $u$.
    \item    [(2)]    $u(n)=n$    if and only if    $j_{n-1}=0$. We denote  $\lambda(u):=0\in Q^\vee$.
    \item    [(3)] Denote     $j_0:=0$. If    $u(n)\neq n$, then the set    $ \{ i\mid j_{i}>0, j_{i-1}=0, 1\leq i\leq n-1 \} $    is not empty,  and we  denote the maximum value of the set as $l$. The permutation  $u$ has
 the reduced expression of form $u=vs_{n-1}\cdots s_{l+1}s_l$    such that
    \begin{enumerate}      \item      $v= u_{j_{n-1}-1}^{(n-2)}\cdots u_{j_l-1}^{(l-1)}u_{j_{l-2}}^{(l-2)}\cdots u_{j_1}^{(1)}=us_{l}s_{l+1}\cdots s_{n-1}$    does not contain the term   $s_{n-1}$;   \item      $\ell(u)-\ell(us_{l}s_{l+1}\cdots s_{n-1})=n-l$;   \item      $\ell(u)-\ell(us_{l-1}s_{l}\cdots s_{n-1})\neq n-l+1$.  \end{enumerate}    Now we define
   \begin{align}\label{lambdau}
     \lambda(u):= \alpha_{l}^{\vee}+\alpha_{l+1}^{\vee}+\cdots + \alpha_{n-1}^{\vee},\mbox{ and } q_{\lambda(u)}:=q_{l}q_{l+1}\cdots q_{n-1}.
  \end{align}
   \end{enumerate}

   \begin{lemma}   \label{lemmaqterm}   Let    $1\leq m\leq n-1$,    $u \in S_{n}$,    $\lambda\in Q^\vee$. Then in    $QH^*(F\ell_n)$, we have
    $N_{s_{n-m}\cdots s_{n-1}, u}^{w, \lambda}\neq 0$    only if    $u(n)\neq n$    and   there exists  $1\leq k\leq n-1$     such that    $\lambda=\alpha_k^\vee+\cdots +\alpha_{n-1}^\vee$.  \end{lemma}

   \begin{theorem}   \label{Seidelelement}   Let   $u \in S_{n}$.  In    $QH^*(F\ell_{n})$, we have
    \begin{align*}
   \sigma^{s_1 s_2 \cdots s_{n-1}}\star \sigma^u &= \left \{
\begin{array}{rl}
 \sigma^{s_1 s_2 \cdots s_{n-1}u},& ~ \text{if } u(n)=n; \\
 q_{\lambda(u)}\sigma^{s_1 s_2 \cdots s_{n-1}u},&  ~   \mbox{if } u(n)\neq n.
\end{array} \right.
\end{align*}
\end{theorem}
We will skip the proof of the above lemma and theorem first, the corresponding details for which  will be given in the next section. The above theorem inspires the following definition.

   \begin{definition}For    $u\in S_n$    and    $k\in\mathbb{Z}_{\geq 0}$, we define
    \begin{align}
     u\uparrow k:=(s_1s_2\cdots s_{n-1})^ku\in S_n,\qquad \lambda(u, k):=\sum_{j=0}^{k-1}\lambda(u\uparrow j).
  \end{align}
 The operator $\mathcal{T}:=\sigma^{s_1 s_2 \cdots s_{n-1}}\star$  is   called  a Seidel operator of    $QH^*(F\ell_n)$.
   \end{definition}
   From  Theorem    \ref{Seidelelement}, for any    $u\in S_n$    and    $k\in\mathbb{Z}_{\geq 0}$, in    $QH^*(F\ell_n)$    we have
    $$\mathcal{T}^k(\sigma^u)=q_{\lambda(u, k)} \sigma^{u\uparrow k}.$$    Note that    $u\uparrow k=u$    if and only if    $n|k$. Therefore  the operator
    $$\widehat{\mathcal{T}}: H^*(F\ell_n)\to H^*(F\ell_n); \quad \widehat{\mathcal{T}}(\sigma^u):=\mathcal{T}(\sigma^u)|_{\boldsymbol{q}=\boldsymbol{1}}=\sigma^{u\uparrow k}$$
    induced by     $\mathcal{T}$  generates a cyclic group   $\mathbb{Z}/n\mathbb{Z}$  action  on   $H^*(F\ell_n)$, which is called the Seidel representation.

   \begin{example}
   The permutation    $u\in S_n$    can be represented by a line of its image:    $u=\overline{u(1), \cdots, u(n)}$. For    $1\leq k\leq n$,    ${\rm id } \uparrow k=\overline{k+1, k+2, \cdots, n, 1, 2, \cdots, k}\in S_n$    exactly corresponds to the maximum partition $(k,k,\cdots,k)$ associated with the Grassmannian $Gr(n-k, n)$. By induction, $\lambda({\rm id}, k)-\lambda({\rm id}, k-1)=\lambda({\rm id} \uparrow (k-1))=\alpha^{\vee}_{n-1}+\alpha_{n-2}^\vee+\cdots+\alpha_{n-k+1}^\vee$. Therefore
    \begin{align}\label{Tact}
     \mathcal{T}^k(\sigma^{\rm id})=q_{n-1}^{k-1}q_{n-2}^{k-2}\cdots q_{n-k+1}\sigma^{{\rm id } \uparrow k},
  \end{align}
  where    $ \sigma^{{\rm id } \uparrow k}=\pi^*(P.D.[{\rm pt}])$    is the image of the highest degree Schubert class    $P.D.[{\rm pt}]$    in    $H^*(Gr(n-k, n))$    under the induced homomorphism    $\pi^*$    of the natural projection map    $\pi: F\ell_n\to Gr(n-k, n)$.
  \end{example}

  As an application of the Seidel operator, we can obtain the following quantum Pieri rule with respect to the  special Schubert class    $\sigma^{s_{n-m}  \cdots s_{n-2}s_{n-1}}$.
    \begin{theorem}   \label{thmQPR}    Let    $1\leq m\leq n-1$    and    $u \in S_{n}$.  Denote    $k:=n-u(n)$.  In  $QH^*(F\ell_{n})$,  we have
    \begin{align}
        \sigma^{s_{n-m}  \cdots s_{n-2}s_{n-1}}\star \sigma^{u}=q_1^{-1}q_{2}^{-2}\cdots q_{n-1}^{1-n}q_{\lambda(u, k)}\mathcal{T}^{n-k}(\sigma^{s_{n-m}  \cdots s_{n-2}s_{n-1}}\cup \sigma^{u\uparrow k}).
     \end{align}
     \end{theorem}

   \begin{proof}
   Observe that    $v(n)=n$ when  $v:=u\uparrow k$. Based on the definition of the Seidel operator, equation \eqref{Tact}, and the commutativity and the associativity of the quantum product $\star $, we have
    $$\mathcal{T}^i(x\star y)=x\star \mathcal{T}^i(y),\qquad \mathcal{T}^n(x)=\mathcal{T}^n({\rm id})\star x=q_1q_2^2\cdots q_{n-1}^{n-1}x,$$
    for any $x, y\in QH^*(F\ell_n)$ and $i\in \mathbb{Z}_{\geq 0}$. Therefore,
    \begin{align*}
    q_1q_2^2\cdots q_{n-1}^{n-1} (\sigma^{s_{n-m}  \cdots s_{n-2}s_{n-1}}\star \sigma^{u})
     &= \mathcal{T}^{n-k}(\sigma^{s_{n-m}  \cdots s_{n-2}s_{n-1}}\star \mathcal{T}^k\sigma^{u}) \\
     &=  \mathcal{T}^{n-k}(\sigma^{s_{n-m}  \cdots s_{n-2}s_{n-1}}\star q_{\lambda(u, k)}\sigma^{v}) \\
      &=  \mathcal{T}^{n-k}(\sigma^{s_{n-m}  \cdots s_{n-2}s_{n-1}}\cup q_{\lambda(u, k)}\sigma^{v}),
 \end{align*}    the last equality in which is obtained by Lemma   \ref{lemmaqterm}.

 \end{proof}

    \begin{example}In    $S_5$,    $u=s_{2}s_{3}s_{4}s_{1}s_{2}s_{3}s_{1}=\overline{43512}$.  So    $k=5-u(5)=3$    and    $\lambda(u\uparrow 0)=\lambda(u)=\alpha_3^\vee+\alpha_4^\vee$.  We have    $u\uparrow 1=s_1s_2s_3s_4s_{2}s_{3}s_{4}s_{1}s_{2}s_{3}s_{1}=s_{3}s_{4}s_{2}s_{3}s_{1}s_{2}s_1$, so    $\lambda(u\uparrow 1)=\alpha_1^\vee+\alpha_2^\vee+\alpha_3^\vee+\alpha_4^\vee$;
   $u\uparrow 2=s_1s_2s_3s_4s_{3}s_{4}s_{2}s_{3}s_{1}s_{2}s_1= s_{4}s_{3}s_{2}$, so    $\lambda(u\uparrow 2)=\alpha_2^\vee+\alpha_3^\vee+\alpha_4^\vee$   ; hence    $u\uparrow 3=s_1=\overline{21345}$,    $q_{\lambda(u, 3)}=q_{\lambda(u\uparrow 0)+\lambda(u\uparrow 1)+\lambda(u\uparrow 2)}=q_1q_2^2q_3^3q_4^3$. By Theorem    \ref{thmQPR}, in    $QH^*(F\ell_5)$, we have
   \begin{align*}
          \sigma^{s_2s_3s_4}\star \sigma^{u}
         &=q_1^{-1}q_2^{-2}q_3^{-3}q_4^{-4}q_{\lambda(u, 3)}\mathcal{T}(\sigma^{s_2s_3s_4} \cup \sigma^{u\uparrow 3})) \\
         &=q_{4}^{-1} \mathcal{T}^2(\sigma^{s_2 s_3s_4} \cup \sigma^{s_1})    \\
         &=q_{4}^{-1} \mathcal{T}^2(\sigma^{s_2 s_3s_4s_1} + \sigma^{s_1 s_2 s_3 s_4})  \\
         &=q_{4}^{-1}  \mathcal{T}(q_{4}\sigma^{s_3 s_4s_1s_2s_3s_1}+ q_{4}\sigma^{ s_2 s_3 s_4s_1s_2 s_3}) \\
         &= q_3q_4 \sigma^{ s_4s_2s_3s_1s_2s_1 }+ q_3q_4 \sigma^{ s_3 s_4s_2s_3s_1s_2 }.
      \end{align*}     \end{example}

   \begin{problem}   \label{problemQH}    For    $u\in S_{n}, 1 \leq i < j <n-1$ , in    $QH^*(F\ell_n)$ , what are the necessary and sufficient conditions for    $\sigma^{s_i s_{i+1} \cdots s_{j}}\star \sigma^{u}=\sigma^{s_i s_{i+1} \cdots s_{j}}\cup \sigma^{u}$    (in terms of combinatorial information of $u$)?  \end{problem}

   \subsection{Proof of Theorem    \ref{Seidelelement}     }    We start by proving Lemma    \ref{lemmaqterm},
    \begin{proof}[ Proof of Lemma \ref{lemmaqterm} ] Note that    $\sigma^{s_{n-m}\cdots s_{n-1}}$    appears in the       quantum product    $(\sigma^{s_{n-1}})^{m}=\sigma^{s_{n-1}}\star\cdots \star \sigma^{s_{n-1}}$ ($m$ copies). By the non-negativity of the quantum Schubert structure constant, we have
    $N_{s_{n-m}\cdots s_{n-1}, u}^{w, \lambda}\neq 0$    only if    $q_\lambda\sigma^w$    appears in    $(\sigma^{s_{n-1}})^{m}\star \sigma^u$.  According to the quantum Chevalley formula (Proposition    \ref{quanchevalley}   ),
    $\lambda=\sum_{j=1}^{n-1}a_j\alpha_j^\vee$    satisfies (i)    $a_{n-1}\geq 1$   ; (ii) for any    $1\leq j<n-1$,    $0\leq a_j\leq a_{n-1}$, and if    $a_j\neq 0$    then    $a_{j+1}\neq 0$.  In particular, if $a_{n-1}=1$, then there exists  $1\leq l\leq n-1$    such that   $\lambda=\alpha_k^\vee+\cdots +\alpha_{n-1}^\vee$.

 If    $a_{n-1}= 2$, we denote   $j_0:=0$    and    $j_{\min}:=\max \{ j\mid a_{j-1}<2, a_j=2, 1\leq j\leq n-1 \} $. If    $a_{j_{\min}-1}=0$, then    $\lambda=\sum_{j=1}^{j_{\min}-2}a_{j}\alpha_{j}^{\vee}+2\alpha^\vee_{j_{\min}}+2\alpha^\vee_{j_{\min}+1}+\cdots+2\alpha^\vee_{n-1}$. Obviously
    $\mbox{sgn}_{\alpha_{j_{\min}}}(w)+\langle \alpha_{j_{\min}}, \lambda\rangle  >  \mbox{sgn}_{\alpha_{j_{\min}}}(s_{n-m}\cdots s_{n-1})+\mbox{sgn}_{\alpha_{j_{\min}}}(u)$. By Proposition    \ref{quantum-classical}    a), we have   $N_{s_{n-m}\cdots s_{n-1}, u}^{w, \lambda}= 0$.  Then we consider the case   $a_{j_{\min}-1}=1$. Note that
    $\langle \alpha_{j_{\min}}, \lambda\rangle=1$ and $\mbox{sgn}_{\alpha_{j_{\min}}}(s_{n-m}\cdots s_{n-1})=0$, hence we obtain
    $\mbox{sgn}_{\alpha_{j_{\min}}}(w)+\langle \alpha_{j_{\min}}, \lambda\rangle  \geq    \mbox{sgn}_{\alpha_{j_{\min}}}(s_{n-m}\cdots s_{n-1})+\mbox{sgn}_{\alpha_{j_{\min}}}(u)$, and  $N_{s_{n-m}\cdots s_{n-1}, u}^{w, \lambda}\neq 0$    only if the equality holds. That is,    $\ell(ws_{j_{\min}})>\ell(w)$    and $\ell(us_{j_{\min}})<\ell(u)$. By Proposition    \ref{quantum-classical}    a), we have
    $N_{s_{n-m}\cdots s_{n-1}, u}^{w, \lambda}=N_{s_{n-m}\cdots s_{n-1}, us_{j_{\min}}}^{ws_{j_{\min}}, \lambda-\alpha_{j_{\min}}^\vee}$.  By induction,    $N_{s_{n-m}\cdots s_{n-1}, u}^{w, \lambda}\neq 0$    only if a series of inequalities about    $u, w$    hold and
    $N_{s_{n-m}\cdots s_{n-1}, u}^{w, \lambda}=N_{s_{n-m}\cdots s_{n-1},  u'}^{w', \lambda-\alpha_{j_{\min}}^\vee-\cdots-\alpha_{n-2}^\vee}$.  However,    $\langle\alpha_{n-1}, \lambda-\alpha_{j_{\min}}^\vee-\cdots-\alpha_{n-2}^\vee\rangle =3$. Using Proposition    \ref{quantum-classical}    a) again,  we know that $N_{s_{n-m}\cdots s_{n-1},  u'}^{w', \lambda-\alpha_{j_{\min}}^\vee-\cdots-\alpha_{n-2}^\vee}$ must be equal to    $0$.  Therefore, we always obtain
    $N_{s_{n-m}\cdots s_{n-1}, u}^{w, \lambda}=0$.

 If    $a_{n-1}\geq 3$ , then    $\mbox{sgn}_{\alpha_{n-1}}(w)+\langle \alpha, \lambda\rangle \geq 0+3>2\geq \mbox{sgn}_{\alpha_{n-1}}(s_{n-m}\cdots s_{n-1})+\mbox{sgn}_{\alpha_{n-1}}(u)$. By Proposition    \ref{quantum-classical}    a), we have    $N_{s_{n-m}\cdots s_{n-1}, u}^{w, \lambda}= 0$.

 For    $\lambda=\alpha_k^\vee+\cdots +\alpha_{n-1}^\vee$, we assume    $u(n)=n$.   If $k<n-1$, we repeat the argument  for the  case of  $a_{n-1}=2, a_{j_{\min}-1}=1$.
  Then  $N_{s_{n-m}\cdots s_{n-1}, u}^{w, \lambda}\neq 0$    only if a series of inequalities about    $u, w$    hold and     $N_{s_{n-m}\cdots s_{n-1}, u}^{w, \lambda}=N_{s_{n-m}\cdots s_{n-1}, u'}^{w', \alpha_{n-1}^\vee}$.  Since    $u'(n)=us_ks_{k+1}\cdots s_{n-2}(n)=u(n)=n$ ,    $\mbox{sgn}_{\alpha_{n-1}}(u')=0$. Thus
    $\mbox{sgn}_{\alpha_{n-1}}(w')+\langle \alpha_{n-1}, \alpha_{n-1}^\vee\rangle \geq 0+2>1= \mbox{sgn}_{\alpha_{n-1}}(s_{n-m}\cdots s_{n-1})+\mbox{sgn}_{\alpha_{n-1}}(u')$.  By Proposition    \ref{quantum-classical}    a), we obtain    $N_{s_{n-m}\cdots s_{n-1}, u'}^{w', \alpha_{n-1}^\vee}= 0$.  \end{proof}

   \begin{lemma}   \label{lemmacupprod}    In    $H^*(F\ell_{n})$, for any    $u \in S_{n}$, we have
    \begin{align*}
   \sigma^{s_1 s_2 \cdots s_{n-1}}\cup \sigma^u &= \left \{
\begin{array}{rl}
 \sigma^{s_1 s_2 \cdots s_{n-1}u},& ~ \text{if } u(n)=n; \\
 0,&  ~   \mbox{if } u(n)\neq n.
\end{array} \right.
\end{align*}     \end{lemma}
    \begin{proof}We have the reduced expression   $u=u_{j_{n-1}}^{(n-1)}\cdots u_{j_2}^{(2)}u_{j_1}^{(1)}$.

 Fristly, consider the case of    $u(n)=n$. This means   $j_{n-1}=0$.    We observe that    $N_{s_1s_2\cdots s_{n-1}, u}^{w, 0}\neq 0$     only if
$\ell(w)=\ell(u)+n-1$  and $s_1s_2\cdots s_{n-1}\leq w$ holds with the Bruhat order. Thus there exists a subsequence of length $n-1$  with product $s_1s_2\cdots s_{n-1}$ in  reduced expression $w=u_{i_{n-1}}^{(n-1)}\cdots u_{i_1}^{(1)}$.
Therefore $i_{n-1}=n-1$.  Note that for any $1\leq k\leq n-2$,
    $\mbox{sgn}_{\alpha_k}(s_1\cdots s_{n-1})=0$. Without loss of generality, suppose $i_1=1$, that is,   $\mbox{sgn}_{\alpha_1}(w)=1$.

   If $\mbox{sgn}_{\alpha_1}(u)=0$, by Proposition \ref{quantum-classical} a), we have $N_{s_1s_2\cdots s_{n-1}, u}^{w, 0}= 0$. Otherwise $\mbox{sgn}_{\alpha_1}(u)=1$ (that is, $\ell(us_1)=\ell(u)-1$), as a consequence of Corollary \ref{cored}, $N_{s_1s_2\cdots s_{n-1}, u}^{w, 0}=  N_{s_1s_2\cdots s_{n-1}, us_1}^{ws_1, 0}$.
   By induction with $\ell(w)-n+1$, $N_{s_1s_2\cdots s_{n-1}, u}^{w, 0}=N_{s_1s_2\cdots s_{n-1}, \rm id}^{wu^{-1}, 0}$ when $\ell(wu^{-1})=\ell(w)-\ell(u^{-1})$ holds, otherwise  $N_{s_1s_2\cdots s_{n-1}, u}^{w, 0}=0$. If $\ell(wu^{-1})=\ell(w)-\ell(u^{-1})$ and $N_{s_1s_2\cdots s_{n-1}, \rm id}^{wu^{-1}, 0} \neq 0$, then $wu^{-1}=s_1\cdots s_{n-1}$. We conclude
   $\sigma^{s_1 s_2 \cdots s_{n-1}}\cup \sigma^u =\sigma^{s_1 s_2 \cdots s_{n-1}u}$.

 If    $u(n)\neq n$, then    $j_{n-1}>0$.  By Corollary  \ref{cored}, we obtain $\sigma^{u_{j_{n-1}}^{(n-1)}}\cup \sigma^{u_{j_{n-2}}^{(n-2)}\cdots u_{j_2}^{(2)}u_{j_1}^{(1)}}=\sigma^u+\mbox{ other terms}$.
 Due to the nonnegativity of Schubert structure constant, the nonzero term of
$\sigma^{s_1 s_2 \cdots s_{n-1}}\cup \sigma^u $ will appear in parts of the  nonzero terms of  $(\sigma^{s_1s_2\cdots s_{n-1}}\cup (\sigma^{u_{j_{n-1}}^{(n-1)}}\cup \sigma^{u_{j_{n-2}}^{(n-2)}\cdots u_{j_2}^{(2)}u_{j_1}^{(1)}})$.
 Due to the induced injective homomorphism $H^*(Gr(n-1, n))\to H^*(F\ell_n)$, we have $\sigma^{s_1s_2\cdots s_{n-1}}\cup \sigma^{u_{j_{n-1}}^{(n-1)}}=0$. Therefore $\sigma^{s_1 s_2 \cdots s_{n-1}}\cup \sigma^u=0$.
   \end{proof}

   \bigskip

   \begin{proof}[ Proof of Theorem \ref{Seidelelement}]If    $u(n)=n$ , by Lemma    \ref{lemmaqterm},    $\sigma^{s_1\cdots s_{n-1}}\star \sigma^u=\sigma^{s_1\cdots s_{n-1}}\cup \sigma^u$    has no non-zero quantum terms. By Lemma    \ref{lemmacupprod},   it is equal to    $\sigma^{s_1\cdots s_{n-1}u}$.

If    $u(n)\neq n$ , by Lemma    \ref{lemmacupprod}    and Lemma    \ref{lemmaqterm},    $\sigma^{s_1\cdots s_{n-1}}\star \sigma^u$    has only quantum terms  $q_\lambda\sigma^w$ of the form    $\lambda=\alpha_k^\vee+\cdots +\alpha_{n-1}^\vee$. By Proposition    \ref{quantum-classical},    we have
   $N_{s_1 s_2 \cdots s_{n-1}, u}^{w, \lambda}\neq 0$     only if    $\ell(ws_{k}s_{k+1}\cdots s_{n-2}s_{n-1})-\ell(w)=n-k$    and
   $\ell(u)-\ell(us_{k}s_{k+1}\cdots s_{n-1})=n-k$.  Then we have
   $$ N_{s_1 s_2 \cdots s_{n-1}, u}^{w, \lambda}=N_{s_1 s_2 \cdots s_{n-1}, us_{k}s_{k+1}\cdots s_{n-1}}^{ws_{k}s_{k+1}\cdots  s_{n-1},0 }.$$    In particular, we have    $\ell(us_k)=\ell(u)-1$ , so    $k\geq l$.  Note that    $us_{k}s_{k+1}\cdots s_{n-1}(n)=u(k)=n$   if and only if     $k=l$.  As a consequence of  Lemma    \ref{lemmacupprod},
   $N_{s_1 s_2 \cdots s_{n-1}, us_{k}s_{k+1}\cdots s_{n-1}}^{ws_{k}s_{k+1}\cdots  s_{n-1},0 }\neq 0$    if and only if    $k=l$    and    $ws_{l}s_{l+1}\cdots s_{n-2}s_{n-1}=s_1 s_2 \cdots s_{n-1}us_{l}s_{l+1}\cdots s_{n-2}s_{n-1}$. Therefore, we have    $$N_{s_1 s_2 \cdots s_{n-1}, u}^{w, \lambda}=\begin{cases}
  1,&\mbox{if } \lambda=\lambda(u) \mbox{ and ~} w=s_1\cdots s_{n-1}u, \\
  0,&\mbox{otherwise.}
\end{cases}$$     \end{proof}

   \subsection{Discussion at Quantum    $K$-theory level  }

The    $K$-theory    $K(G/P)$    of the flag varieties   $G/P$    is a Gothendieck group generated by the isomorphism class    $[E]$    composed of algebraic vector bundles on  $G/P$. The additive structure and multiplicative structure  of  $K(G/P)$  are given $[E]+[F]:=[E \oplus F]$ and $[E]\cdot [F]:=[E \otimes F]$, respectively. We simply denote the Schubert class   $\mathcal{O}^w_P:=[\mathcal{O}_{X^w}]$,  and note $K(G/P)=\bigoplus_{w\in W^P}\mathbb{Z}\mathcal{O}^w$. In general, the quantum product of the quantum    $K$    theory is a formal power series with respect to  quantum parameters.  But for the quantum $K$-theory $QK(G/P)$ of flag varieties , Anderson-Chen-Tseng \cite{ACT} showed that the quantum product of any two Schubert classes is still a polynomial in the  quantum variable. Recall   $\Delta_P=\Delta\setminus\{\alpha_{n_1}, \cdots, \alpha_{n_k}\}$. As a consequence, $QK(G/P)=(K(G/P)\otimes \mathbb{C}[q_{n_1},\cdots, q_{n_k}], *)$. The quantum product
    $$\mathcal{O}^{u} \ast \mathcal{O}^{v}=\sum_{w\in W^P,  \lambda_P \in Q^\vee/Q^\vee_P)}\kappa_{u,v}^{w,\lambda}q_{\lambda_P} \mathcal{O}^{w}, $$    is determined by the structure constants   $\kappa_{u,v}^{w,\lambda_P}$,  which is (complicated and signed) combination of $K$-theoretic genus-zero 3-pointed (and 2-pointed) Gromov-Witten invariants (see e.g. \cite{BuMi, BCLM} for more details). We notice the following facts:
    $$\kappa_{u,v}^{w,\lambda_P}=N_{u,v}^{w,\lambda_P}\quad \mbox{holds whenever }\quad\ell(u)+\ell(v)=\ell(w)+\langle c_1(G/P), \lambda_P\rangle.$$
    Combining the properties and applications of the Seidel operator in the quantum    $K$-theory of the Grassmannian   \cite{LLSY, BCP23}    with the discussion in this section on    $QH^*(F\ell_n)$,  we conjecture that there are the same expressions in quantum $K$-theory about Theorem
 \ref{Seidelelement} and Theorem \ref{thmQPR}.

    \begin{conjecture}   \label{conjQK}
   Let    $1\leq m\leq n-1$    and    $u \in S_{n}$. Denote    $k:=n-u(n)$.  In    $QK(F\ell_{n})$,  we  have
    \begin{align}
      \mathcal{T}(\mathcal{O}^u):=\mathcal{O}^{s_{1}  \cdots s_{n-2}s_{n-1}} * \mathcal{O}^{u}&=q_{\lambda(u)} \mathcal{O}^{u\uparrow 1}, \\
    \mathcal{O}^{s_{n-m}  \cdots s_{n-2}s_{n-1}}* \mathcal{O}^{u}&=q_1^{-1}q_{2}^{-2}\cdots q_{n-1}^{1-n}q_{\lambda(u, k)}\mathcal{T}^{n-k}(\mathcal{O}^{s_{n-m}  \cdots s_{n-2}s_{n-1}}\cdot  \mathcal{O}^{u\uparrow k}).
     \end{align}
     \end{conjecture}

Assuming that the above conjecture is correct,  we can  obtain  parts of  quantum products of the quantum    $K$-theory  by calculating   classical products in  $K(G/P)$, similar to the case of quantum cohomology. Furthermore, for any    $w\in W$, there is a unique    $(w', w'')\in W^P\times W_P$    such that    $w=w'w''$. In    \cite{Kato}, Kato proved that there are  the following ``functoriality" in quantum $K$ theory of flag vaieties $G/P$  with the natural projection map   $\pi: G/B\to G/P$    (see also \cite{BCLM} for the special case $G/P\to G/G={\rm pt}$).
    \begin{proposition}[Theorem A of \cite{Kato}]   \label{surj}    The following map is surjective algebra homomorphism:
    $$\pi_*: QK(G/B)\to QK(G/P);\quad  \pi_*(\mathcal{O}^w_B)=\mathcal{O}_P^{w'},\quad \pi_*(q_i)=\begin{cases}
      1,&\mbox{if }\alpha_i\in \Delta_P, \\
      q_i,&\mbox{if }\alpha_i\notin \Delta_P.
    \end{cases}$$     \end{proposition}
    For    $QK(G/B)$, we abbreviate the Schubert class    $\mathcal{O}^u=\mathcal{O}^u_B$.

    \begin{example} Consider    $QK(F\ell_4)$ and  let $u=s_{2}s_{3}s_{2}s_{1}$. By Conjecture    \ref{conjQK},
    $\mathcal{T}^2(\mathcal{O}^{u})=q_1q_2q_{3} \mathcal{T}(\mathcal{O}^{s_3})=q_1q_2q_{3}^{2} \mathcal{O}^{s_1s_2}$, and    $$\mathcal{O}^{s_2 s_3} \ast \mathcal{O}^{s_1s_2}=\mathcal{O}^{s_2 s_3} \cdot \mathcal{O}^{s_1s_2}=\mathcal{O}^{s_2 s_3s_1s_2} + \mathcal{O}^{s_1 s_2 s_3 s_2}-\mathcal{O}^{s_2s_1s_3s_2s_3},$$    where the last equality is obtained by Pieri's rule for the classical $K$-theory  \cite{LeSo}. As a consequence,
    \begin{align*}
          \mathcal{O}^{s_2s_3}\ast \mathcal{O}^{u}
         &=q_1^{-1}q_{2}^{-2}q_{3}^{-3} q_1q_2q_3^2 \mathcal{T}^2(\mathcal{O}^{s_2s_3} \cdot \mathcal{O}^{s_1s_2}) \\
         &=q_{2}^{-1}q_{3}^{-1} \mathcal{T}^2(\mathcal{O}^{s_2 s_3s_1s_2} + \mathcal{O}^{s_1 s_2 s_3 s_2}-\mathcal{O}^{s_2s_1s_3s_2s_3})  \\
         &= q_{2}^{-1}q_{3}^{-1}\mathcal{T}(q_2 q_3 \mathcal{O}^{s_3 s_2s_1}+ q_{2}q_3 \mathcal{O}^{ s_2 s_1 s_3}-q_{2}q_3 \mathcal{O}^{ s_3 s_2 s_1 s_3}) \\
         &=  {q_1q_2q_3\mathcal{O}^{\rm  id }+ q_3 \mathcal{O}^{ s_1 s_3s_2s_1 }}-q_1q_2q_3\mathcal{O}^{ s_3 }.
      \end{align*}
      \end{example}

   \begin{example}   \label{exevidence}    In    $QK(F\ell_6)$, for  $u=s_5 s_3 s_4 s_{1}s_{2}s_{3} s_2s_1$,  we have  $\mathcal{T}(\mathcal{O}^{u})=q_1q_2q_{3}q_4q_5 \mathcal{O}^{s_4s_2s_3 }$.  Based on   Conjecture    \ref{conjQK}    and the classical Pieri rule, we have
    \begin{align*}
          \mathcal{O}^{s_3s_4s_5}\ast \mathcal{O}^{u}
         &=q_1^{-1}q_{2}^{-2}q_{3}^{-3}q_4^{-4}q_5^{-5} q_1q_2q_3q_4q_5 \mathcal{T}^5(\mathcal{O}^{s_3s_4s_5} \cdot \mathcal{O}^{s_4s_2s_3}) \\
         &= q_{2}^{-1}q_{3}^{-2}q_4^{-3}q_5^{-4}\mathcal{T}^5 (\mathcal{O}^{s_1s_2 s_3s_4s_5s_3} +\mathcal{O}^{s_2s_3s_4s_5s_4s_3} +\mathcal{O}^{ s_2 s_3 s_4s_5s_2s_3} +\mathcal{O}^{s_3s_4s_5s_4s_2s_3}  \\
     &=  {q_1q_2q_3q_4q_5  \mathcal{O}^{s_3 }+  \mathcal{O}^{ s_1 s_2 s_3s_4s_5s_3s_4s_2s_3s_2s_1}+\mathcal{O}^{ s_1 s_2 s_3s_4s_5s_4s_1s_2s_3s_2s_1}+\mathcal{O}^{  s_2 s_3s_4s_5s_3s_4s_1s_2s_3s_2s_1}}  \\
         &\qquad -q_1q_{2}q_3q_4q_5 \mathcal{O}^{ s_4 s_3}
         - q_1q_2q_3q_4q_5\mathcal{O}^{ s_2s_3 }-2\mathcal{O}^{ s_1 s_2 s_3s_4s_5s_3s_4s_1s_2s_3s_2s_1}+ q_1q_2q_3q_4q_5\mathcal{O}^{ s_4s_2s_3 }.
      \end{align*}

    \ytableausetup{boxsize=0.3em}
     We can use Young tableaux to express the partitions for $Gr(3, 6)$. For example, the Young tableau ${ } \ydiagram{3,2,0}$ represents partition $(3, 2, 0)$, which corresponds to the Grassmannian type permutation $\overline{146235}$. By Proposition \ref{surj}, for $\pi_*: QK(F\ell_{6}) \longrightarrow  QK(Gr(3,6))$, we have
   \begin{align*}
       \mathcal{O}^{\ydiagram{1,0,0}} \ast \mathcal{O}^{\ydiagram{3,2,1}}&=\mathcal{O}^{s_3}_P \ast \mathcal{O}^{s_5s_3s_4s_1s_2s_3}_P \\
       &=\pi_*(\mathcal{O}^{s_3s_4s_5}) \ast \pi_*(\mathcal{O}^u) \\
        &=\pi_*(\mathcal{O}^{s_3s_4s_5}  \ast \mathcal{O}^u) \\
        &=\pi_*(q_1q_2q_3q_4q_5  \mathcal{O}^{s_3 }+  \mathcal{O}^{ s_1 s_2 s_3s_4s_5s_3s_4s_2s_3s_2s_1}+\mathcal{O}^{ s_1 s_2 s_3s_4s_5s_4s_1s_2s_3s_2s_1}+\mathcal{O}^{  s_2 s_3s_4s_5s_3s_4s_1s_2s_3s_2s_1}  \\
         &\qquad -q_1q_{2}q_3q_4q_5 \mathcal{O}^{ s_4 s_3}
         - q_1q_2q_3q_4q_5\mathcal{O}^{ s_2s_3 }-2\mathcal{O}^{ s_1 s_2 s_3s_4s_5s_3s_4s_1s_2s_3s_2s_1}+ q_1q_2q_3q_4q_5\mathcal{O}^{ s_4s_2s_3 }) \\
         &=q_3  \mathcal{O}^{s_3 }_P+  \mathcal{O}^{ s_1 s_2 s_3s_4s_5s_3s_4s_2s_3}_P+\mathcal{O}^{ s_1 s_2 s_3s_4s_5s_4s_1s_2s_3}_P+\mathcal{O}^{  s_2 s_3s_4s_5s_3s_4s_1s_2s_3}_P  \\
         &\qquad -q_3 \mathcal{O}^{ s_4 s_3}_P
         - q_3\mathcal{O}^{ s_2s_3 }_P-2\mathcal{O}^{ s_1 s_2 s_3s_4s_5s_3s_4s_1s_2s_3}_P+ q_3\mathcal{O}^{ s_4s_2s_3 }_P \\
         &=q_3  \mathcal{O}^{\ydiagram{1,0,0} }+ \mathcal{O}^{ \ydiagram{3,3,1}}+\mathcal{O}^{ \ydiagram{3,2,2} }+\mathcal{O}^{ \ydiagram{3,3,2}}
         -q_3 \mathcal{O}^{ \ydiagram{2,0,0}}
         - q_3 \mathcal{O}^{ \ydiagram{1,1,0} }-2\mathcal{O}^{ \ydiagram{3,3,2}}+ q_3\mathcal{O}^{ \ydiagram{2,1,0} }.
      \end{align*}    Abbreviate the quantum variable of  $QK(Gr(3, 6))$   as    $q=q_3$,  we have
    \begin{align*}
      \mathcal{O}^{\ydiagram{1,0,0}} \ast \mathcal{O}^{\ydiagram{3,2,1}}    &=q \mathcal{O}^{\ydiagram{1,0,0}}-q\mathcal{O}^{\ydiagram{2,0,0}}-q \mathcal{O}^{\ydiagram{1,1,0}}+q\mathcal{O}^{\ydiagram{2,1,0}}+\mathcal{O}^{\ydiagram{3,2,2}}
         +\mathcal{O}^{\ydiagram{3,3,1}}-\mathcal{O}^{\ydiagram{3,3,2}},
      \end{align*}   which  is consistent with the calculation   obtained by using the quantum Pieri rule for    $QK(Gr(3, 6))$    in    \cite{BuMi}. This can be seen as an evidence of Conjecture    \ref{conjQK}.  \end{example}

We end this section with a question similar to   Problem \ref{problemQH}.
   \begin{problem}   \label{problemQK}    For    $u\in S_{n}, 1 \leq i < j <n-1$, in    $QK(F\ell_n)$,  what are the necessary and sufficient conditions for    $\mathcal{O}^{s_i s_{i+1} \cdots s_{j}}*\mathcal{O}^{u}=\mathcal{O}^{s_i s_{i+1} \cdots s_{j}}\cdot \mathcal{O}^{u}$    (in terms of combinatorial information of  $u$)? Is it consistent with the necessary and sufficient conditions at the quantum cohomology level?
   \end{problem}


\begin{thebibliography}{99}

\bibitem{ACT}Anderson D, Chen L, Tseng H H.
{\it On the finiteness of quantum K-theory of a homogeneous space}, Int  Math  Res  Not,  2022, 1313-1349

\bibitem{Belk}   Belkale P.  {\it Transformation formulas in quantum cohomology}, Compos Math, 2004, 140: 778-792

\bibitem{Bertram1997} Bertram A. {\it Quantum Schubert calculus}, Adv Math,  1997, 128: 289-305

\bibitem{BGG1973} Bernstein I N,  Gel'fand I M,  Gel'fand S I. {\it Schubert cells and cohomology of the spaces G/P}, Russian Mathematical Surveys,   1973, 28: 1-26

 \bibitem{BrGr}   Bryan J,  Graber T. {\it The crepant resolution conjecture},
In: Algebraic Geometry-Seattle 2005. Part 1, Proc Sympos Pure Math, vol. 80. Providence: Amer Math  Soc,  2009, 23-42


  \bibitem{BCP23}Buch A S,   Chaput P E,    Perrin N. {\it Seidel and Pieri products in cominuscule quantum K-theory},  preprint at arXiv: math.AG/2308.05307
 \bibitem{BCLM} Buch  A S, Chung S, Li  C, Mihalcea L. {\it Euler characteristics in the quantum K-theory of flag
varieties}, Selecta Math,  2020,  26: Article No. 29, 11 pp

\bibitem{BKT2003} Buch A S, Kresch A,  Tamvakis, H.
{\it Gromov-Witten invariants on Grassmannians},
J Amer Math Soc, 2003, 16: 901-915

\bibitem{BKT2} Buch A S, Kresch A,  Tamvakis H. {\it Quantum Pieri rules for isotropic Grassmannians}, Invent  Math, 2009, 178: 345-405

\bibitem{BuMi} Buch A,   Mihalcea L. {\it Quantum         $K$-theory of Grassmannians}, Duke Math  J,  2011,  156: 501-538

 \bibitem{CMP}Chaput P E,    Manivel L,   Perrin N. {\it Quantum cohomology of minuscule homogeneous spaces. II. Hidden symmetries}, Int Math Res Not,  2007, Art. ID rnm107, 29 pp

\bibitem{CLLZ}  Chen B, Li A M, Li  X, Zhao G.
{\it Ruan's conjecture on singular symplectic flops of mixed type},
Sci  China Math,  2014,  57: 1121-1148

\bibitem{Ciocan1999} Ciocan-Fontanine I. {\it On quantum cohomology rings of partial flag varieties}, Duke Math J, 1999, 98: 485-524

\bibitem{CIJ18}  Coates T,   Iritani H,  Jiang Y. {\it  The crepant transformation conjecture for toric complete intersections},
Adv  Math,  2018,  329: 1002-1087

\bibitem{CIT09}  Coates T,   Iritani H,  Tseng H H.
{\it Wall-crossings in toric Gromov-Witten theory I: Crepant examples}, Geom  Topol,  2009,  13: 2675-2744



\bibitem{CR13}  Coates T,   Ruan Y.
{\it Quantum cohomology and crepant resolutions: a conjecture},
Ann  Inst  Fourier (Grenoble),   2013,  63: 431-478


\bibitem{FP1997} Fulton W, Pandharipande R.
{\it Notes on stable maps and quantum cohomology}, In: Proceedings of Symposia in Pure Mathematics, 62, Part 2. Providence: Amer Math Soc, 1997, 45-96

\bibitem{FW2004}   Fulton W,  Woodward C.
{\it On the quantum product of Schubert classes},
J Algebraic Geom, 2004, 13: 641-661

\bibitem{GoWo}Gonz\'{a}lez  E, Woodward  C T. {\it A wall-crossing formula for Gromov-Witten invariants under variation of git quotient},  Math  Ann,  2024, 388: 4135-4199
\bibitem{LiHu15} Huang Y, Li C. {\it On equivariant quantum Schubert calculus for          $G/P$},  J  Algebra,   2015, 441: 21-56

 \bibitem{Hump75}  Humphreys J E.  {\it  Linear algebraic groups},    Graduate Texts in Mathematics 21, Springer-Verlag, New York-Berlin,   1975

\bibitem{Ir09} Iritani H.
{\it An integral structure in quantum cohomology and mirror symmetry for toric orbifolds},
Adv  Math,  2009, 222: 1016-1079

\bibitem{Kato} Kato  S. {\it On quantum K-group of partial flag manifolds },  preprint at arXiv: math.AG/1906.09343

  \bibitem{KT2003}  Kresch A,  Tamvakis H. {\it Quantum cohomology of the Lagrangian Grassmannian}, J Algebraic Geom, 2003, 12: 777-810
       \bibitem{KT2004} Kresch A,  Tamvakis H. {\it Quantum cohomology of orthogonal Grassmannians}, Compos Math,  2004, 140: 482-500
\bibitem{LLW} Lee Y P, Lin H W, Wang C L.
{\it Flops, motives, and invariance of quantum rings},
Ann  of Math (2),  2010, 172: 243-290

 \bibitem{LeSo} Lenart  C, Sottile F. {\it A Pieri-type formula for the          $K$-theory of a flag manifold},
Trans  Amer  Math  Soc,  2007, 359: 2317-2342

 \bibitem{LL2010} Leung N C,  Li C. {\it Functorial relationships between         $QH^{*}(G/B)$         and         $QH^{*}(G/P)$}, J Differential Geom, 2010, 86: 303-354

 \bibitem{LL2012} Leung N C, Li C.
{\it Classical aspects of quantum cohomology of generalized flag varieties},
Int Math Res Not, 2012, 16: 3706-3722
\bibitem{LL2013} Leung N C, Li C.
{\it Quantum Pieri rules for tautological subbundles}, Adv. Math. 248 (2013), 279-307


\bibitem{Li15} Li C. {\it Functorial relationships between          $QH^*(G/B)$          and          $QH^*(G/P)$         (II)},  Asian J  Math, 2015, 19: 203-234

\bibitem{LLSY} Li C, Liu Z, Song J, Yang M. {\it On  Seidel representation in quantum         $K$-theory of   Grassmannians},  preprint at arXiv: math.AG/2211.16902
\bibitem{NaSa} Naito S, Sagaki D.
{\it Pieri-type multiplication formula for quantum Grothendieck polynomials}, preprint at  arXiv: math.QA/2211.01578

 \bibitem{Peterson1997}   Peterson D. {\it Quantum cohomology of         $G/P$},  Lecture notes at MIT, 1997 (notes by J. Lu and K. Rietsch)

\bibitem{Post} Postnikov A. {\it  Affine approach to quantum Schubert calculus},  Duke Math J, 2005, 128: 473-509

 \bibitem{Ruan}Ruan Y. {\it The cohomology ring of crepant resolutions of orbifolds}, In: Gromov-Witten Theory of Spin Curves and Orbifolds. Contemp  Math, vol. 403. Providence: Amer  Math  Soc, 2006, 117-126

\bibitem{Seid}  Seidel P.     {\it     $\pi_1$         of symplectic automorphism groups and invertibles in quantum homology rings},
Geom  Funct  Anal,  1997, 7: 1046-1095

\bibitem{Sottile} Sottile F. {\it Pieri's formula for flag manifolds and Schubert polynomials},
Ann  Inst  Fourier (Grenoble),  1996, 46: 89-110


\bibitem{Woodward2005} Woodward T.
{\it On D. Peterson's comparison formula for Gromov-Witten invariants of G/P},
Proc Amer Math Soc, 2005, 133: 1601-1609



\end{thebibliography}
\end{document}